\documentclass[11pt] {llncs}
\usepackage[letterpaper, margin=1.3in]{geometry}
\usepackage{amssymb,amsmath}
\usepackage{algorithm,algorithmic}
\usepackage{graphicx}
\usepackage{caption}
\usepackage{subfig}

\newtheorem{defn}{Definition}
\newtheorem{conjectures}{Conjecture}
\newtheorem{facts}{Fact}

\title{Improved Upper Bounds on $a'(G\Box H)$}

\author{Punit Mehta$^{1}$, Rahul Muthu$^{1}$, Gaurav Patel$^{1}$,\\ Om Thakkar$^{2}$, Devanshi Vyas$^{1}$}
\institute{Dhirubhai Ambani Institute of Information \& Communication Technology, Gandhinagar, India$^{1}$, \\ The Pennsylvania State University, State College, USA$^{2}$ \\ \email{\{mehta\_punit,rahul\_muthu,201101065,201101141\}@daiict.ac.in,omthkkr@cse.psu.edu}}
\date{}
\begin{document}
\maketitle
\begin{abstract} The acyclic edge colouring problem is extensively studied in graph theory. The corner-stone of this field is a conjecture of Alon et. al.\cite{alonacyclic} that $a'(G)\le \Delta(G)+2$. In that and subsequent work, $a'(G)$ is typically bounded in terms of $\Delta(G)$. Motivated by this we introduce a term $gap(G)$ defined as $gap(G)=a'(G)-\Delta(G)$. Alon's conjecture can be rephrased as $gap(G)\le2$ for all graphs $G$. In \cite{manusccartprod} it was shown that $a'(G\Box H)\le a'(G)+a'(H)$, under some assumptions. Based on Alon's conjecture, we conjecture that $a'(G\Box H)\le a'(G)+\Delta(H)$ under the same assumptions, resulting in a strengthening. The results of \cite{alonacyclic} validate our conjecture for the class of graphs it considers. We prove our conjecture for a significant subclass of sub-cubic graphs and state some generic conditions under which our conjecture can be proved. We suggest how our technique can be potentially applied by future researchers to expand the class of graphs for which our conjecture holds. Our results improve the understanding of the relationship between Cartesian Product and acyclic chromatic index. 
\end{abstract}

{\bf keywords}: Acyclic edge colouring$\cdot$Acyclic chromatic index$\cdot$Cartesian Product $\cdot$Gap
\section{Introduction}
An acyclic edge colouring of a graph is an assignment of colours to
its edges such that adjacent edges get distinct colours and the edge
set of any cycle uses at least three distinct colours. This is a widely
studied variant of graph colouring. The acyclic chromatic index of a
graph, denoted $a'(G)$ is the least number of colours used in any
acyclic edge colouring of the graph. It was conjectured by Alon,
Sudakov and Zaks \cite{alonacyclic} that for all graphs
$a'(G)\le\Delta(G)+2$, where $\Delta(G)$ represents the maximum
degree. There have been many results on the acyclic edge colouring
problem, but the problem is rich enough to retain many difficult and
interesting unsolved variants. Some of the works in this field are \cite{iplacycedcolpap},\cite{nearopti},\cite{gridlike}.

In this paper we give some results which show the effect of cartesian product on the acyclic chromatic index of the resulting graph in terms of the individual factor graphs. The results in this paper build on \cite{manusccartprod} and \cite{gridlike}. We define a new vertex chromatic index and show how bounding its value can be used to extend the results of the previous papers to a wider class of graphs. Our results (and more importantly our plausible looking conjecture) are very significant because they extend the understanding of the long standing problem of acyclic edge colouring in the well known and wide framework of cartesian products of graphs. Our conjecture if correct and proved, will imply that Alon's Conjecture is true except possibly for prime graphs under the cartesian product operation. Looking at the history of research conducted in the field, we believe this is a significant stride forward.

In Section~\ref{secdef} we define some basic terms. In Section~\ref{relres} we present the most relevant results impacting our work. In Section~\ref{secconj} we give details of the ideas used in the proof of our conjecture assuming a lemma on a new type of vertex colouring. In Section~\ref{newchromatic} we define a new vertex chromatic number of graphs and use it to prove some cases of our conjecture. In Section~\ref{subcubic} we show that our conjecture is true for the cartesian product of a large class of graphs with a large subclass of subcubic graphs. In Section~\ref{secspecgrap} we give much tighter results for cartesian products of any graph with the Petersen Graph or $K_4$. Finally in Section~\ref{concl} we indicate some concrete future directions of research indicated by our work. 

\section{Definitions and Notation}\label{secdef}

We present here the operation of the cartesian product of graphs.
\begin{defn}
Given two graphs $G_1 = (V_1,E_1)$ and $G_2 = (V_2,E_2)$, the {\em
cartesian product} of $G_1$ and $G_2$, denoted $G_1 \Box G_2$, is
defined to be the graph $G=(V,E)$ where $V=V_1 \times V_2$ and $E$
contains the edge joining $(u_1,u_2)$ and $(v_1,v_2)$ if and only if
{\em either} $u_1=v_1$ and $(u_2,v_2) \in E_2$ {\em or} $u_2=v_2$ and
$(u_1,v_1) \in E_1$.
\end{defn}

The reader may verify that the maximum degrees of graphs under the
cartesian product satisfies the following relation. 

\begin{facts}\label{cartproddeg}
\[\Delta(G\Box H)=\Delta(G)+\Delta(H)\]
\end{facts}

We now state Alon's conjecture from \cite{alonacyclic}

\begin{conjectures}\label{alonconj}
$a'(G)\le \Delta(G)+2, \forall G$.
\end{conjectures}

Since the acyclic chromatic index $a'(G)$ is usually bounded in terms of $\Delta(G)$, we introduce the following useful concept. We define the $gap(G)$ of a graph $G$ as the difference between its acyclic chromatic index and its maximum degree. Thus, \[gap(G)=a'(G)-\Delta(G)\]
Thus the conjecture of Alon et. al. is essentially that: \[\forall G, gap(G)\le 2\]

We use (a,b) to denote a nontrivial maximal bichromatic path using the colours $a$ and $b$ on its edges. We use XY to denote an edge between vertices X and Y. These notations are used for the proof mentioned in Section \ref{subcubic}.

\section{Related results}\label{relres}

The main result of \cite{gridlike}, due to Muthu, Narayanan and Subramanian, giving optimal acyclic edge colourings of grid like graphs (graphs
obtained by taking the cartesian product of an arbitrary number of
graphs each of which is either a single edge, a longer path or a
cycle) follows from the following general theorem also derived in that same paper. Here we only state this lemma. For the consequent results (which are tight for most of the cases and off by
at most one in a few remaining cases) as well as the proof of that lemma the reader is advised to refer to that paper.

\begin{theorem}\label{thmgrid} Let G be a simple graph with $a'(G)=\eta$. Then,
\begin{enumerate}
\item  $a'(G\Box P_2)\le \eta+1$, if $\eta\ge2$.
\item  $a'(G\Box P_l)\le\eta+2$, if $\eta\ge2$ and $l\ge3$. 
\item  $a'(G\Box C_l)\le \eta+2$, if $\eta>2$ and $l\ge3$.
\end{enumerate}
\end{theorem}

The following result \cite{manusccartprod} due to Muthu and Subramanian is weaker, but has a broader scope.
\begin{theorem}\label{cartprod}
Let $G=(V_G,E_G)$ and $H=(V_H,E_H)$ be two connected nontrivial
graphs, such that $max\{a'(G),a'(H)\}>1$. Then, \[a'(G\Box H)\le a'(G)+a'(H)\] 
\end{theorem}

The technique used to prove this can be outlined as follows. Initially we assume that we have optimal acyclic edge colourings of $G$ and $H$ available. It is known that the cartesian product graph can be viewed as $n(H)$ disjoint copies of $G$ each corresponding to a vertex of $H$, and a perfect matching between the corresponding vertices of two copies of $G$ which correspond to adjacent vertices of $H$. We refer to these two classes of edges as local edges and crossing edges respectively. (We also use the reference: edges of $G$ and edges of $H$, respectively).(See Fig.1)
 
\begin{figure}
\centering
\includegraphics[width = 50mm]{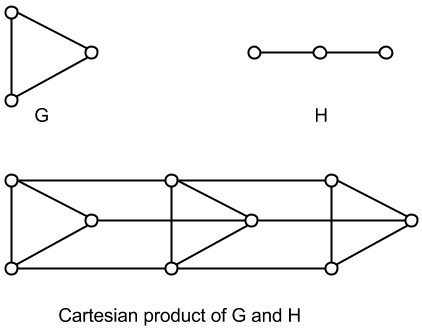}
\caption{}
\end{figure}

In our colouring scheme, the crossing edges between a pair of copies of $G$ all get the same colour as the corresponding single edge in an optimal acyclic edge colouring of $H$. The different copies of $G$ receive isomorphic colourings to each other, meaning that the partition of the edge set into colour classes are identical across the copies. However we use a set of such colourings, where the names of the colour classes are permuted with respect to each other. No two such permutations assigns the same colour for the edges corresponding to the same member of the partition. We call such permutations {\bf nonfixing permutations}. This underlying technique used in all cases of our proof here are closely related to the proof for cartesian product of arbitrary graphs with $P_2$ in \cite{gridlike}.  Clearly, this constrains the number of such permutations to be at most the number of colours in the optimal acyclic edge colouring of $G$. 

Since we begin with acyclic edge colourings of $G$ and $H$ respectively, using disjoint sets of colours, the overall colouring is clearly proper.  If there are any bichromatic cycles created after the colouring, they must use edges of $G$ and $H$ since each is individually acyclic. These are eliminated by using nonfixing permutations of the colouring of $G$ for any two copies of $G$ which would result in bichromatic cycles prior to this step. As is established later in this paper, pairs of copies of $G$ liable to bichromatic cycles prior to this step are either adjacent to each other or correspond to endpoints of maximal bichromatic paths in $H$.

\begin{figure}
    \centering
    \subfloat[Without using nonfixing permutations - bichromatic cycles are created]{{\includegraphics[width=90mm]{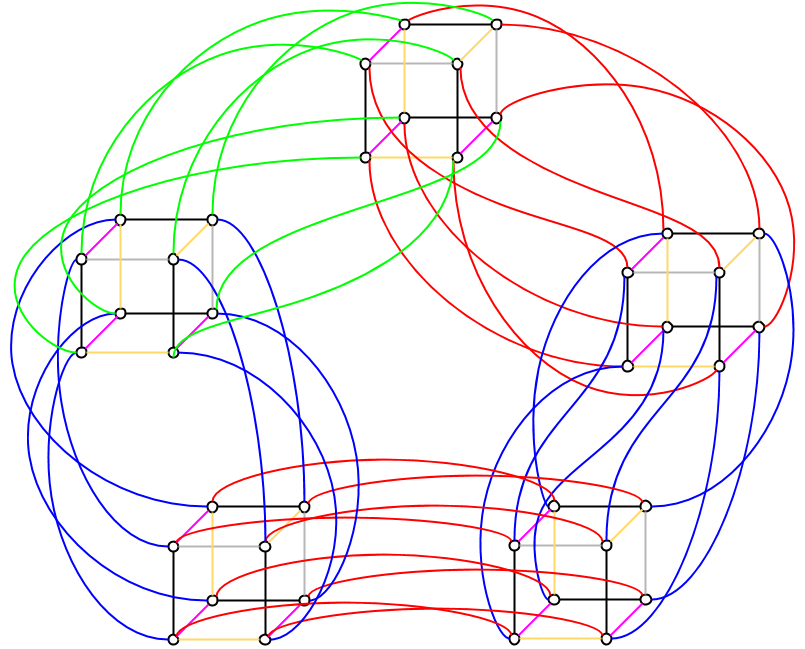} }}%
    \qquad
    \subfloat[Using nonfixing permutations]{{\includegraphics[width=90mm]{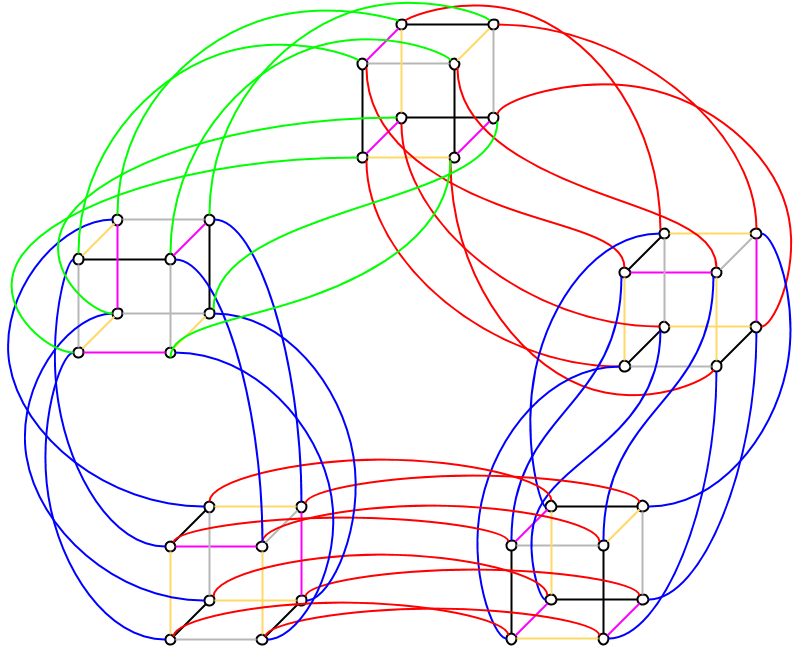} }}%
    \caption{}
    \label{fig:permutation}
\end{figure}


%

We assume that $a'(G)\ge a'(H)$. The idea is that $a'(H)\ge\Delta(H)\ge \chi(H)$ with the exception of $H$ being a complete graph or an odd cycle. In case $H$ is an odd cycle or a complete graph, it is well known that $a'(H)\ge \Delta(H)+1=\chi(H)$.  This enables us to colour the copies of $G$ according to the vertex colouring in an optimal vertex colouring of $H$ and give nonfixing permutations of the same basic colouring of $G$ to copies which receive distinct colours. Nonfixing permutations means that the two colourings are the same in terms of partition of the edge set into colour classes, but the set of edges receiving any particular colour are different from the set of edges receiving the same colour in the two copies. The importance (and in fact indispensability) of this idea can be seen from the following two figures ($2(a)$ \& $2(b)$), the first of which does not use nonfixing permutations while the second does. We see how the bichromatic cycles which are present in the first colouring are eliminated in the second while properness is retained without creating any new bichromatic cycles. The requirement $a'(G)\ge a'(H)$ is to ensure we have adequate nonfixing permutations. A set of $k$ elements cannot have more than $k$ mutually non-fixing permutations over it. 

Viewing Conjecture~\ref{alonconj} in conjunction with Theorem~\ref{cartprod} indicates that the latter result cannot be optimal when applied to pairs of graphs with total gap more than 2. Thus, we make the following conjecture.

\begin{conjectures}\label{ourconj}
Let $G$ and $H$ be two graphs with gap more than zero each. We, additionally, assume that they are connected and max$\{a'(G),a'(H)\}>1$. Then $a'(G\Box H)\le a'(G)+\Delta(H)$. Note that this result may hold by exchanging the reference of $G$ and $H$ if necessary.
\end{conjectures}

Since, a major class of graphs for which we prove our conjecture in this paper are subcubic graphs, we state here two important earlier results obtained for the acyclic edge colouring of subcubic graphs, by San Skulrattankulchai \cite{iplacycedcolpap}; and Manu \& Chandran\cite{basavaraju2008acyclic} respectively. 
\begin{theorem}
$a'(G)\le5$ for all graphs with $\Delta(G)\le3$.
\end{theorem}
 This result was improved by  to:
 \begin{theorem}
 If $G$ is not 3-regular, then $a'(G)\le 4$ for graphs with $\Delta(G)\le3$.
\end{theorem}   

\section{Our Conjecture}\label{secconj}
Our conjecture, stated in the previous section, is an attempt to improve Theorem~\ref{cartprod}. Clearly if either $G$ or $H$ has $gap=0$, our conjecture is the same as Theorem~\ref{cartprod} and hence valid. In the case when each of them has $gap>0$, the idea we propose is to follow the colouring scheme used in the proof of  Theorem~\ref{cartprod} initially.

This gives rise to two exhaustive cases. The first is a bichromatic cycle formed between adjacent copies. This can be ruled out by using nonfixing permutations for the corresponding local colourings of $G$. See Figures $2(a)$ \& $2(b)$ for the colourings without and with nonfixing permutations respectively. Thus neighbouring copies must get different permutations of colourings. The other case is a bichromatic cycle formed using the edges of copies which correspond to endpoints of a maximal bicromatic path in the acyclic colouring of $H$. Thus copies connected by maximal bichromatic paths should be given relatively nonfixing permutations.

Both Theorem~\ref{cartprod} and Conjecture~\ref{ourconj} are validated for the classes of graphs considered in \cite{gridlike}. Consequently, we recast those results in a more general framework here. We state general conditions which help demonstrate the conjecture for a several families of graphs. We establish our conjecture for the cartesian product of a large class of graphs with a large subclass of subcubic graphs using the results of \cite{iplacycedcolpap} and \cite{basavaraju2008acyclic}. In particular, for the Petersen Graph, we prove our conjecture for its cartesian products with any possible graph. For $K_4$ we prove a slightly weaker result.

\section[conj]{A lemma useful for several cases of our conjecture}\label{newchromatic}

We now define a new type of vertex colouring and an associated new vertex chromatic number, $\chi_{\nu}$. Unlike the usual vertex colouring, our colouring is based on an auxiliary super graph of the given graph, over the same vertex set. We denote this auxiliary graph of $G$ by $aux(G)$. In order to compute this new vertex colouring, we add edges between vertices if they are the endpoints of a maximal bichromatic path in an optimal acyclic edge colouring. The new chromatic number is calculated by minimising this number over all possible optimal acyclic edge colourings of the graph. In the context of the preceding explanations it is clear that we want the copies of $G$ lying at endpoints of a maximal bichromatic path of $H$ to get distinct nonfixing permutations. Thus we need to give them distinct colours and towards this end we add extra edges between such pairs of vertices, provided they are not already adjacent in $G$. 

\begin{figure}
    \centering
    \includegraphics[width=90mm]{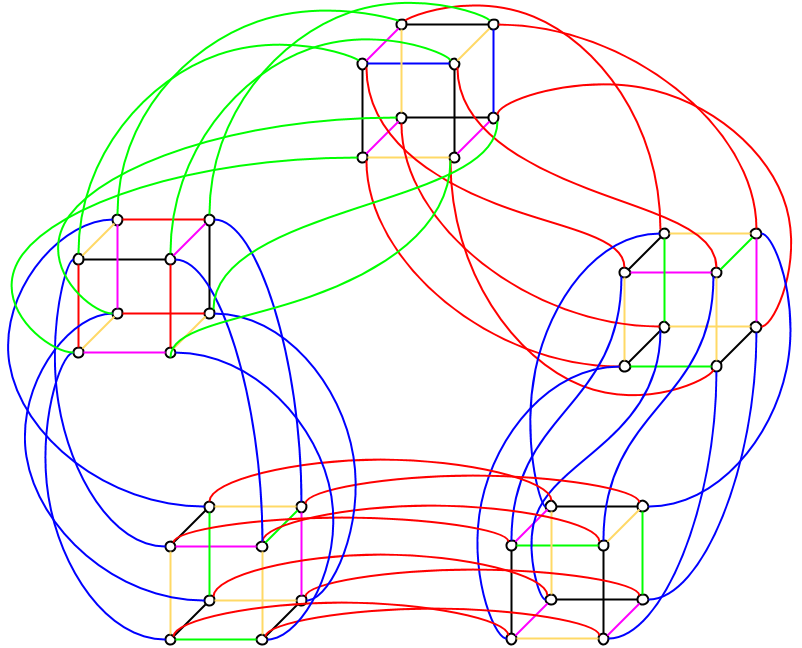}%
    \caption{Reusing the missing color}
    \label{fig:permutation}
\end{figure}

We state and prove a lemma which establishes our conjecture for the cartesian product of pairs of graphs $G$ and $H$, where $a'(G)\ge \chi_{\nu}(H)$.
\begin{lemma}
Let $G$ and $H$ be two connected graphs. Let their acyclic chromatic indices be $a'(G)$ and $a'(H)$ respectively. Now under an optimal acyclic edge colouring of $H$ (with, obviously, $a'(H)$ colours), let the new chromatic number $\chi_{\nu}$ be $k$. Then if $a'(G)\ge k$, then $a'(G\Box H)\le a'(G)+\Delta(H)$.
\end{lemma}
 This is because we will be able to give any two copies of $G$ which are either adjacent or endpoints of a maximal bichromatic path in $H$ distinct nonfixing permutations. Now, at each copy of $G$, the last $gap(H)$ colour classes of that copy of $G$ are rewritten with the $gap(H)$ missing colours at that copy. The proof is clear if $gap(H)=1$, but needs argument for larger values of $gap(H)$. Compare the colourings before and after reusing the missing colours (resulting in a reduction of one colour overall) in Figures $2(b)$ \& $3$ respectively. In fact we do not attempt to provide a proof of this lemma for larger values of $gap$ and leave that as an open problem for us or other researchers to pursue in future. 

Since our conjecture is complex we prove a relaxed version for some classes of graphs. In particular we assume that Alon's conjecture is true and thus limit ourselves to the graphs $G$ and $H$ each having gap 1 or 2. The way we prove the result in special cases, is to assume $H$ to be some fixed graph and with reference to this instance of $H$ we allow $G$ to vary in an unrestricted fashion. 

In order to apply the method based on our new chromatic number, it is necessary that $a'(G)$ is greater than or equal to $\chi_{\nu}(H)$. For the specific graphs we handle, we use special case analysis to handle the situation when $a'(G)$ is less than this threshold. This either involves interchanging the roles of $G$ and $H$ or some other direct methods.

\section{Cartesian Product with subcubic graphs $H$, with $\Delta(H)=3$ and $a'(H)\le4$}\label{subcubic}

\begin{theorem}
Consider a graph $H$ which has maximum degree $\Delta$=3 and $a'(H)=4$. Then,
\[ \chi_{\nu}(H) \leq 6\]
\end{theorem}
\begin{proof}
Since the graph is subcubic, clearly no vertex has degree greater than 3. In this section we are dealing with the case of $a'(H)\le 4$. We use this to compute the number of colours present and absent at any vertex. For a vertex to be the endpoint of a maximal bichromatic path it must involve a colour present at the vertex and a colour absent at the vertex. Thus for degree three vertices, there can be at most three maximal bichromatic paths with it as an endpoint. For degree two vertices, the number of maximal bichromatic paths with it as endpoint can be at most 4. For a degree 1 vertex, the number of maximal bichromatic paths with it as an endpoint is at most 2. Thus in each case the sum of the vertex's degree in the graph and the number of maximal bichromatic paths for which it acts as an endpoint is at most 6. In other words the maximum degree of the auxiliary graph we had defined is at most 6. Hence, the graph $H$ has new chromatic number, $\chi_{\nu}(H)\le6$ unless the resulting auxiliary graph is the complete graph on 7 vertices (Brooks' Theorem~\cite{west}). We now establish why $aux(G)$ cannot be $K_7$.
\end{proof}

\begin{theorem}
Consider a non-regular graph G which has maximum degree $\Delta$=3. Then, $aux(G)$ can never be $K_7$.
\end{theorem}

\begin{proof}
Consider a case when $aux(G)$ is $K_7$ for an optimal acyclic edge colouring $c$ of $G$ and graph $G$ has a vertex of degree 1, say $u$. Such vertex can have degree atmost 3 in $aux(G)$ when $u$ is connected to a degree 3 vertex in $G$. Therefore, if $aux(G)$ is $K_7$, then $G$ can not have a vertex of degree 1. Since, graph $G$ has maximum degree $\Delta=3$ and it does not have a vertex of degree 1, it must have a vertex of degree 2, say $A$. 

\begin{case} $A$ is in a triangle (shown in Fig.4(a)) in $G$. 

Since, $A$ can not be an endpoint of all (1,3), (2,3), (1,4), (2,4) maximal 
bichromatic paths simultaneously, it can not have 6 neighbours in $aux(G)$. Therefore, degree-2 vertices can not be in a triangle in $G$.
\end{case}

\begin{case} $A$ is not in a triangle and its neighbours ($B$ and $C$) are connected by an intermediate vertex $E$ (shown in Fig.4(b)) in $G$. 

Since, $aux(G)$ is $K_7$, $G$ must have 7 vertices. Let $H$ be the $7^{th}$ vertex. Since, $A$ and $H$ can not be adjacent, there must be a maximal bichromatic path $P$ between them. Now, $H$ is an endpoint in $P$ and $H$ can not have neighbour $D$ (or $F$) in $P$ as it would remove the maximal bichromatic path between $A$ and $D$ (or $F$) in $G$. So, $E$ and $H$ must be adjacent in $P$ and hence, in $G$. Now, there can be either (1,4) or (2,3) maximal bichromatic path possible between $A$ and $H$. Considering this fact and without the loss of generality, assume that the edge $EH$ is coloured as 1. Note that, we can not have an edge $DF$ in $G$ as it would remove the maximal bichromatic path between $A$ and $D$ (or $F$) in $G$ because of the colour given to $DF$. Therefore, vertices $D$,$H$ and $F$,$H$ must be adjacent in $G$. $DH$ must be coloured 2 as colouring this edge with 4 would lead to the violation in the maximal bichromatic path between $A$ and $H$. $FH$ must be coloured 3 as it is the only free colour available on $H$ and $F$. 

Now, it can be seen that there can not be a maximal bichromatic path between the vertices $D$ and $F$. So, $aux(G)$ can not be $K_7$ in such a configuration.
\end{case}

\begin{case} $A$ is not in a triangle and its neighbours ($B$ and $C$) are not connected by an intermediate vertex (shown in Fig.4(c)) in $G$.

Note that we can not have edges $DH, EF. DE, HF$ in $G$ as argued in Case: 2. So, $G$ must have edges $DF$ and $EH$. Consider a vertex $D$ which has degree 2 and is connected to another degree 2 vertex. Therefore, $D$ can not be an endpoint of 4 maximal bichromatic paths simultaneously and hence, it can not have 6 neighbours in $aux(G)$. 

This completes the proof.
\end{case}
\end{proof}


\begin{figure}
    \centering
    \subfloat[Configuration for Case 1]{{\includegraphics[width=40mm]{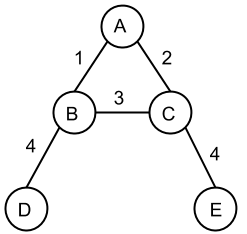} }}%
    \qquad
    \subfloat[Configuration for Case 2]{{\includegraphics[width=40mm]{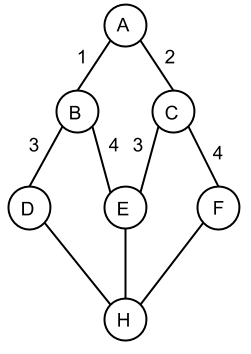} }}%
    \qquad
    \subfloat[Configuration for Case 3]{{\includegraphics[width=40mm]{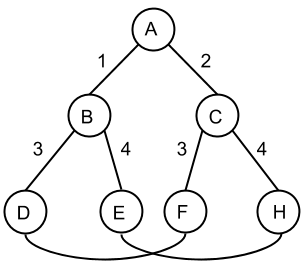} }}%
    \caption{}
    \label{fig:2and3}
\end{figure}

%
%
%
%

We conjecture the following statement on the new chromatic number of subcubic graph.
\begin{conjectures}\label{ourconj}
Let G be a subcubic graph. Then, 
\[ \chi_{\nu}(G) \leq 4\]
\end{conjectures}
\begin{theorem}
Let $G$ be any graph such that $a'(G)\le3$ or $a'(G)\ge6$. Let $H$ be any subcubic graph with $a'(H)\le4$. Then $a'(G\Box H)\le a'(G)+\Delta(H)$.
\end{theorem}
If $a'(H)=\Delta(H)=3$ then the result is true due to Theorem~\ref{cartprod}. If $a'(H)=4$, then when $a'(G)\ge6$ we can use pairwise nonfixing permutations of some optimal colourings of $G$, which corresponds to the new vertex colouring being optimal. Since we know that $\chi_{\nu}(H)\le6$, there are enough nonfixing permutations of the colouring of $G$ when $a'(G)\ge6$. Thus the result is true for this case.
Now suppose $a'(G)\le3$. This implies that $\Delta(G)\le3$. If $\Delta(G)\le2$ then the theorem is true from the results of \cite{gridlike}. If $\Delta(G)=3$ and $a'(G)=3$ then the result follows from Theorem~\ref{cartprod}. This leaves unresolved only the cases when $a'(G)\in\{4,5\}$. This will follow if the conjecture~\ref{ourconj} is true. We are still working on it and have not found any counter examples with extensive computer aided search.   

\begin{figure}
    \centering
    \subfloat[$K_4$]{{\includegraphics[width=30mm]{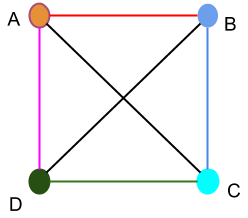} }}%
    \qquad
    \subfloat[Petersen graph]{{\includegraphics[width=50mm]{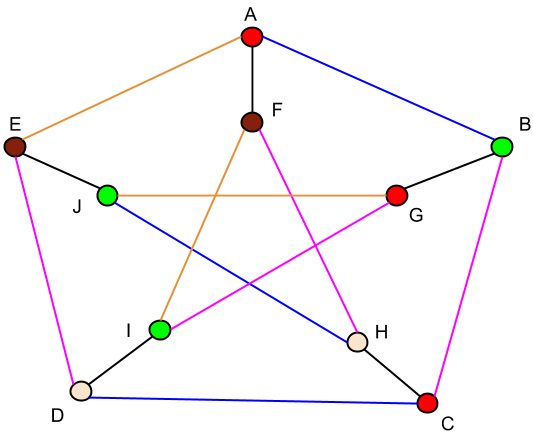} }}%
    \caption{optimal acyclic edge colouring along with optimal $\chi_{\nu}$}
    \label{fig:K4andpetersen}
\end{figure}

\section{Petersen Graph and $K_4$}\label{secspecgrap}
\begin{theorem}
Let $H\in\{$Petersen,$K_4\}$ and let $G$ be any connected graph. Then $a'(G \Box H)\le a'(G)+a'(H)-1$.
\end{theorem}
Here, we improve upon the results of the previous section, for the specific subcubic graphs Petersen and $K_4$ (referred to collectively as $H$). We use $G$ to refer to the other graph involved in the cartesian product. Since most of the ideas are similar to the previous section, we only briefly describe the peculiarities of these graphs which allows this improvement. Specifically, we simply show diagrams with optimal acyclic edge colourings of these graphs along with an optimal vertex colouring of the auxiliary graph we have defined. Thus the proof is self-evident (see Fig.5) for the cases where $a'(G)\ge4$. For the cases when $a'(G)=3$ and $\Delta(G)=3$ then the result follows from Theorem~\ref{cartprod}. If $\Delta(G)\le3$ then the result follows from \cite{gridlike}. This concludes the proof of the theorem.  

\section{Conclusions and future directions}\label{concl}
We have introduced a technique for rewriting one of the colour classes of each copy of $G$ with one  of the colours of $H$ missing at the vertex corresponding to that copy. This follows the phase of using nonfixing permutations for appropriate pairs of copies. The ideas in their current form don't work for $gap(H)=2$ (due to Alon's conjecture, we assume the gap is not larger) and we see scope for resolving this. The other problem is obtaining better estimates on our newly defined chromatic number for subcubic graphs as well as other graphs.
\bibliographystyle{plain} \bibliography{acyc}	
\end{document}